\documentclass[12pt]{amsart}
\usepackage{amssymb}
\usepackage{amsfonts}
\usepackage{epsfig}
\usepackage{amsmath}
\usepackage{graphicx}
\usepackage{caption}
\usepackage{subcaption}
\usepackage{wrapfig}
\usepackage
[       left=1.5in,
        right=1in,
        top=1in,
        bottom=1in,
]
{geometry}

\usepackage{mathtools}

\newtheorem{theorem}{Theorem}[section]

\newtheorem{lemma}[theorem]{Lemma}

\theoremstyle{definition}
\newtheorem{definition}[theorem]{Definition}

\newtheorem{remark}[theorem]{Remark}
\newtheorem{question}[theorem]{Question}

\numberwithin{equation}{subsection}

\newcommand{\bbH}{\mathbb{H}}

\begin{document}

\title{Densities of Hyperbolic Cusp Invariants}
\author{C. Adams, R. Kaplan-Kelly, M. Moore,\\ B. Shapiro, S. Sridhar, J. Wakefield}
\begin{abstract}
We find that cusp densities of hyperbolic knots in $S^3$ are dense in $[0,0.6826\dots]$ and those of links are dense in $[0,0.853\dots]$. We define a new invariant associated with cusp volume, the cusp crossing density, as the ratio between the cusp volume and the crossing number of a link, and show that cusp crossing density for links is bounded above by $3.1263\dots$. Moreover, there is a sequence of links with cusp crossing density approaching 3. The least upper bound for cusp crossing density remains an open question. For two-component hyperbolic links, cusp crossing density is shown to be dense in the interval $[0,1.6923\dots]$ and for all hyperbolic links, cusp crossing density is shown to be dense in $[0, 2.120\dots]$. 
\end{abstract}

\maketitle

\section{Introduction}

\hspace{0.6cm} In order to study hyperbolic knots and links, one considers their complements and the invariants that are associated to them. Two such invariants are volume and cusp volume. The volume of hyperbolic manifolds has been the focus of much study. This paper considers certain invariants related to the cusps. 

A \emph{cusp} of a non-compact finite volume hyperbolic 3-manifold is a submanifold homeomorphic to $T\times [0,1)$ that lifts to a collection of horoballs in hyperbolic 3-space. In the case of knots and links, a cusp is topologically a tubular neighborhood of the link intersected with the link complement. A cusp is \emph{maximal} if there is no larger cusp containing it, which occurs exactly when the cusp is tangent to itself at one or more points. 

The cusp volume $cv(K)$ of a knot $K$ is the volume of its maximal cusp. The \emph{maximal cusp volume} of a link $L$, $cv(L)$, is the sum of the volumes of the individual cusps that are not overlapping in their interiors and that yields the largest possible total. Note that for a two-component link, the maximal cusp volume is realized when one of the two cusps is maximized first and the other is maximized relative to it.

It is natural then to consider the relationship between cusp volume and volume, which motivates the study of their ratio, cusp density.

\begin{definition} 
The \emph{cusp density} of a knot or link $L$, $cd(L)$, is defined as the ratio of the maximal cusp volume $cv(L)$ to the volume of the complement $vol(L)$. 

$$cd(L):=\frac{cv(L)}{vol(L)}$$ 
\end{definition}

Let $v_{tet} = 1.01494\dots$ be the volume of an ideal regular tetrahedron in hyperbolic 3-space and $v_{oct} = 3.6638\dots$ the volume of an ideal regular octahedron in hyperbolic 3-space. In \cite{Bor}, B{\"o}r{\"o}czky proved that the densest packing of horoballs in hyperbolic space gives a ratio of $\frac{\sqrt{3}}{2v_{tet}}\approx 0.853\dots$ which in turn provides an upper bound on cusp density (cf.\cite{Mey2}). This value for cusp density is achieved by the figure-eight knot. By results of \cite{Reid}, this is the only knot that achieves this highest possible cusp density. In \cite{Adams6}, it was proved that cusp densities for all cusped finite volume hyperbolic 3-manifolds are dense in the interval $[0,0.853\dots]$.

In Section 2, we find intervals in which the respective cusp densities of knots and links are dense. We first outline the construction from \cite{Adams6} for general cusped hyperbolic 3-manifolds, and then show that it in fact proves that cusp densities of link complements in $S^3$ are a dense subset of 
$[0, 0.853 \dots]$. We next construct intervals in which cusp densities of knots are dense, given a knot with a particular cusp density, and then use explicit examples to obtain the following theorem:

\begin{theorem}
Cusp densities of knots in $S^3$ are dense in $[0,0.6826\dots]$.
\end{theorem}

A crucial construction for our proof is \emph{belted sums} of links as shown in Figure \ref{beltedsum}. In \cite{Adams7}, it was proved that because incompressible twice-punctured disks are totally geodesic with a unique hyperbolic structure,  the hyperbolic structures of the two link complements are preserved when they are glued together along the cut open twice punctured disks bounded by the trivial component. Thus, the volume of the belted sum is the sum of the two respective volumes, as in Figure \ref{beltedsum}. In this paper, we study the effect of belted sums on cusp volume, and use these results to prove the main lemma we need. For this purpose, the following definition will be useful.

\begin{definition} 
Let $\mathcal{A}$ be a subset of components of a link $L$. The \emph{restricted cusp density} of $\mathcal{A}$ in $L$, $cd_R(L, \mathcal{A})$, is defined to be the ratio of the maximal  cusp volume in $\mathcal{A}$, denoted $cv_R(L,\mathcal{A})$, to the volume of the complement.

$$cd_R(L, \mathcal{A}):=\frac{cv_R(L, \mathcal{A})}{vol(L)}$$ 
\end{definition}

The results on cusp volumes for belted sums rely on finding two links $L_1$ and $L_2$, one having high restricted cusp density of a single cusp and the other having low restricted cusp density. Unfortunately, finding suitable links with restricted cusp density greater than $.6826\dots$ has so far been unsuccessful. It is the only obstacle to showing that cusp densities of knots are dense in a larger interval than $[0, .6826\dots]$. It remains an open question as to whether there exists a sequence of knots with cusp densities approaching $0.853\dots$ .

In Section 3 we define and study a new link invariant which we call cusp crossing density. The invariant \emph{volume density}, $d_{vol}(L)$, which is the ratio of the volume to the crossing number, has been considered in \cite{Adams10}, \cite{Bur}, \cite{CKP1}, and \cite{CKP2}. Volume densities of all hyperbolic knot and link complements lie in $[0,3.6638\dots]$ and are dense in that interval. One crucial motivation for studying volume density has been to explore interesting relations with the determinant density for knots, which is defined to be \[d_{det}(K) = \frac{2\pi \ln(det(K))}{c(K)}\]  This is  a combinatorial invariant of knots that is dense in the same interval $[0,3.6638\dots]$. 

We define cusp crossing density analogously to volume density.

\begin{definition}
Let $L$ be a hyperbolic link in $S^3$ and $c(L)$ denote its crossing number. The \emph{cusp crossing density}, $d_{cc}(L)$ is defined as 

$$d_{cc}(L):=\frac{cv(L)}{c(L)}$$
\end{definition} 

We look at bounds for cusp crossing density. Previously, in \cite{ACFGK}, it was proved that for any knot $K$, $cv(K) \leq \frac{9c}{2} (1-1/c)^2$. Hence $d_{cc}(K) < 4.5$.

We  show that in fact, cusp crossing densities for hyperbolic knots and links are bounded above by $3.1263\dots$.  We then obtain several families of links with cusp crossing density approaching 3 from below. We find sequences of hyperbolic knots with cusp crossing density as high as 1.706. We also show that cusp crossing density for hyperbolic two-component links is dense in the interval $[0,1.6923\dots]$ and cusp crossing density for all hyperbolic links is dense in the interval $[0, 2.120\dots]$. We suspect that the actual upper bound on cusp crossing density for links is 3 and that cusp crossing densities of links are dense in the interval $[0,3]$.

The results in this paper would not have been possible without the opportunity to experiment afforded by the computer program SnapPy \cite{snap}. We are very grateful to the creators of that program.

\begin{figure}[h]
\centering
\begin{subfigure}{.3\textwidth}
\centering
\includegraphics[scale=.5]{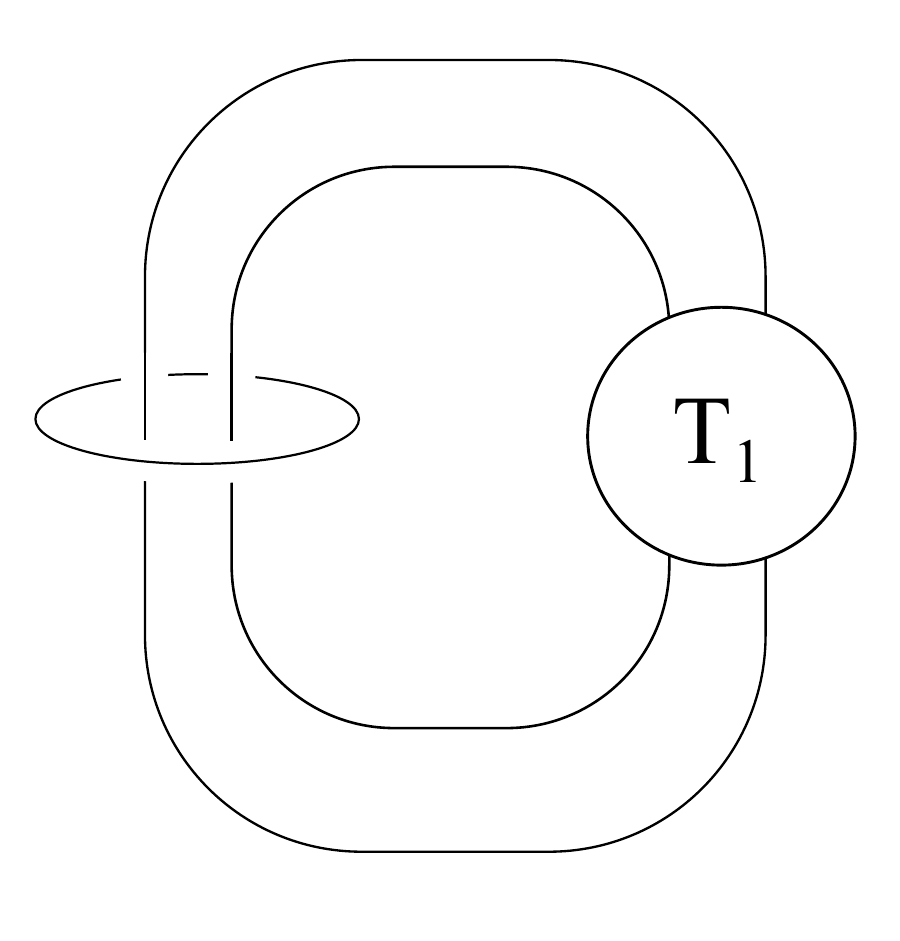}
\end{subfigure}
\begin{subfigure}{.3\textwidth}
  \centering
  \includegraphics[scale=.5]{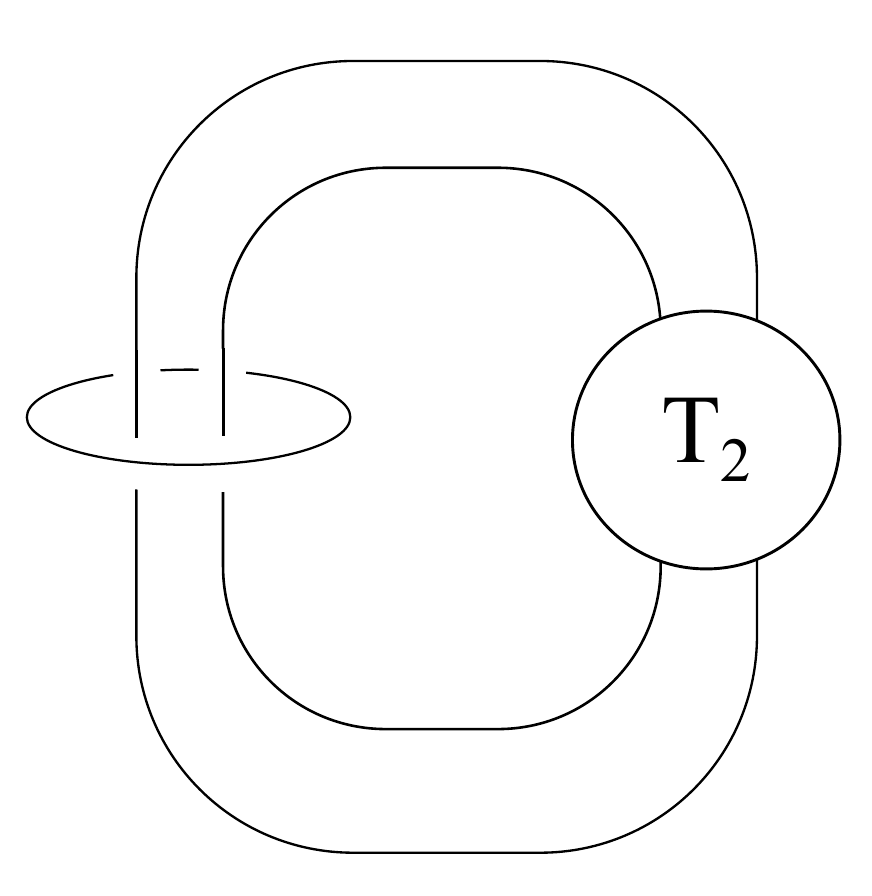}
  \caption{Link $L_2$}
\end{subfigure}
\begin{subfigure}{.3\textwidth}
  \centering
 \includegraphics[scale=.5]{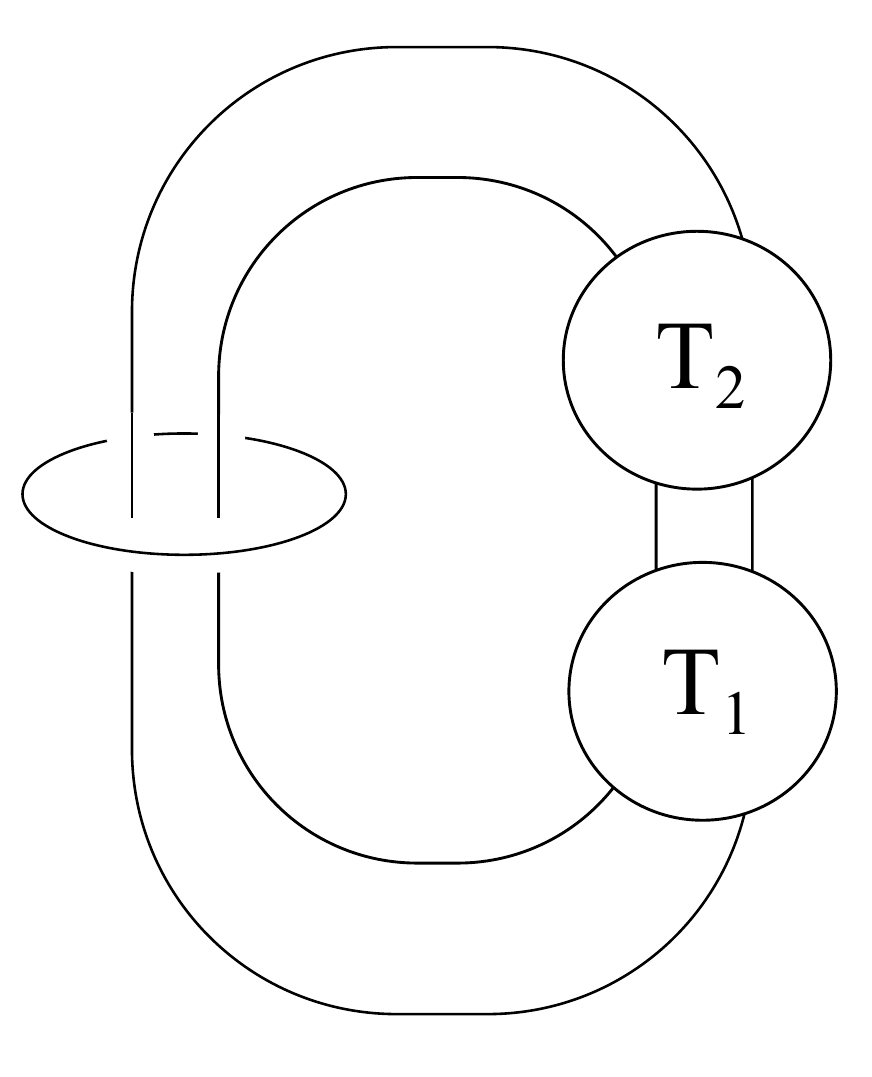}
  \caption{Link $L_3$}
\end{subfigure}
\caption{A belted sum of two links}
\label{beltedsum}
\end{figure}

\section{Cusp Density of Links}

Let $L$ be a link and $A$  a trivial component of $L$ bounding an incompressible $n-$punctured disk.  Consider a projection of $L$ such that $n$ strands pass through the disk bounded by $A$ (see Figure 2).  

\begin{figure}[h]
\centering
\includegraphics[scale=.5]{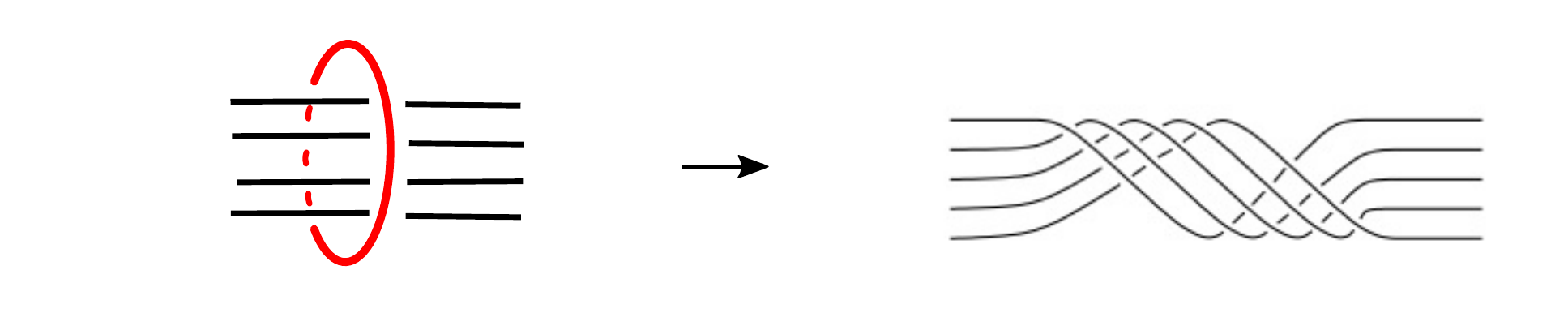}
\caption{(1,1)-Dehn filling on trivial component}
\label{twist}
\end{figure}

\begin{remark} 
 A $(1,p)$-Dehn filling on component $A$ gives the complement of the link $L'$ obtained from $L$ by removing $A$ and adding $p$ full twists to the $n$ strands that passed through it, as shown in Figure 2.  
\end{remark}


\begin{definition}
Let $\mathcal{A}$ be a subset of components of link $L$, such that all components of $\mathcal{A}$ are trivial and any pair of components in $\mathcal{A}$ form either the Hopf link or the unlink.  Define the graph $G$ with the components in $\mathcal{A}$ as vertices and edges between pairs of components that are linked.  We will call $\mathcal{A}$ a {\it chain of shape $G$ in $L$}.
\end{definition}

\begin{lemma}
\label{chain lemma}
Let $L$ be a link in $S^3$ of at least $m+1$ components containing a chain $\mathcal{A} = \{A_1, A_2, \dots, A_m\}$ of shape $G$.  If $G$ has no cycles, then the manifold obtained by $(1,p_i)$-Dehn filling each of the components $A_i$ in $L$ remains a link complement.
\end{lemma}

\begin{proof} 
If $G$ has no cycles, then it can be considered as a union of finitely many disconnected trees. If $G$ consists of finitely many trees, each with $n$ or fewer edges, then induction on $n$ will show that $(1,p)$-Dehn filling the components in $\mathcal{A}$ yields a link complement.

If $n = 0$ then $G$ has only isolated vertices, so none of the components of $\mathcal{A}$ are linked with each other. This implies that they can then be filled one by one to get a link complement, and the twists from each filling in the resulting link will be isolated and not affect each other.

Now assume the lemma holds when all trees have less than or equal to $n - 1$ edges, and let the trees of $G$ have $n$ or fewer edges.  For each nontrivial tree in $G$, choose a leaf vertex of the tree.  Each corresponding component is only linked with the component that is its parent in its tree and possibly some components of $L$ that are not in $\mathcal{A}$. 

As shown in Figure \ref{twist}, $(1,p_i)$-Dehn filling on one of the leaf components $A_i$ gives $p_i$ full twists to the strands passing through the disk it bounds, but its parent component $A_j$, which only passes through the disk once,  remains trivial and all other components in $\mathcal{A}$ are unaffected. The Dehn filling for $A_j$ becomes  a $(1, p_j + p_i)$-Dehn filling.  By induction, filling on the remaining components will result in a link complement.
\end{proof}

\begin{theorem}
The set of all cusp densities for hyperbolic link complements in $S^3$ is dense in the interval $[0, 0.853\dots].$
\end{theorem}

\begin{proof}  
In \cite{Adams6}, it was proved that the set of cusp densities of hyperbolic manifolds is dense in the interval $[0, 0.853\dots]$ by constructing manifolds with cusp density arbitrarily close to any value in the interval.  A manifold with cusp density arbitrarily close to $x \in [0, 0.853\dots]$ is constructed by first finding a link complement $F_{k,n,m}$ with restricted cusp density arbitrarily close to $x$, where restricted cusp density is the ratio of the cusp volume of a subset $\mathcal{C}$ of the components to the volume of the manifold.  When all components of the link except for those in $\mathcal{C}$ are $(1,p)$-Dehn filled, where $p$ can be chosen to be arbitrarily large, the volume and cusp volume of the resulting manifold approach the volume and restricted cusp volume of the original manifold, respectively.  The resulting manifold then has cusp density arbitrarily close to $x$.

It remains to be shown that the manifolds obtained by this construction are in fact link complements.  Let $D_n$ be the $n-$component alternating daisy chain with even $n > 4$, and let $\mathcal{C}_0 = \{C_1,C_2,C_3,C_4\}$ be a set of 4 cusps in $D_n$ such that $C_1,C_2,C_3$ are adjacent and $C_4$ is opposite $C_2$.  Let $D_{n,m}$ be the link obtained by taking an $m-$fold cyclic cover of $D_n$ around $C_4$, and let $\mathcal{C}_1$ be the set of cusps in $D_{n,m}$ that cover those in $\mathcal{C}_0$.  $F_{k,n,m}$ is the belted sum of $D_{n,m}$ and $L_k$, a $k-$fold cover of a link obtained by taking covers of the minimally twisted 5-chain, where the belt in $D_{n,m}$ is a component covering $C_2$.  Let $\mathcal{C}$ be the union of $\mathcal{C}_1$ and the set of all cusps in $L_k$.

\begin{figure}[h]
\centering
 \includegraphics[scale =.3]{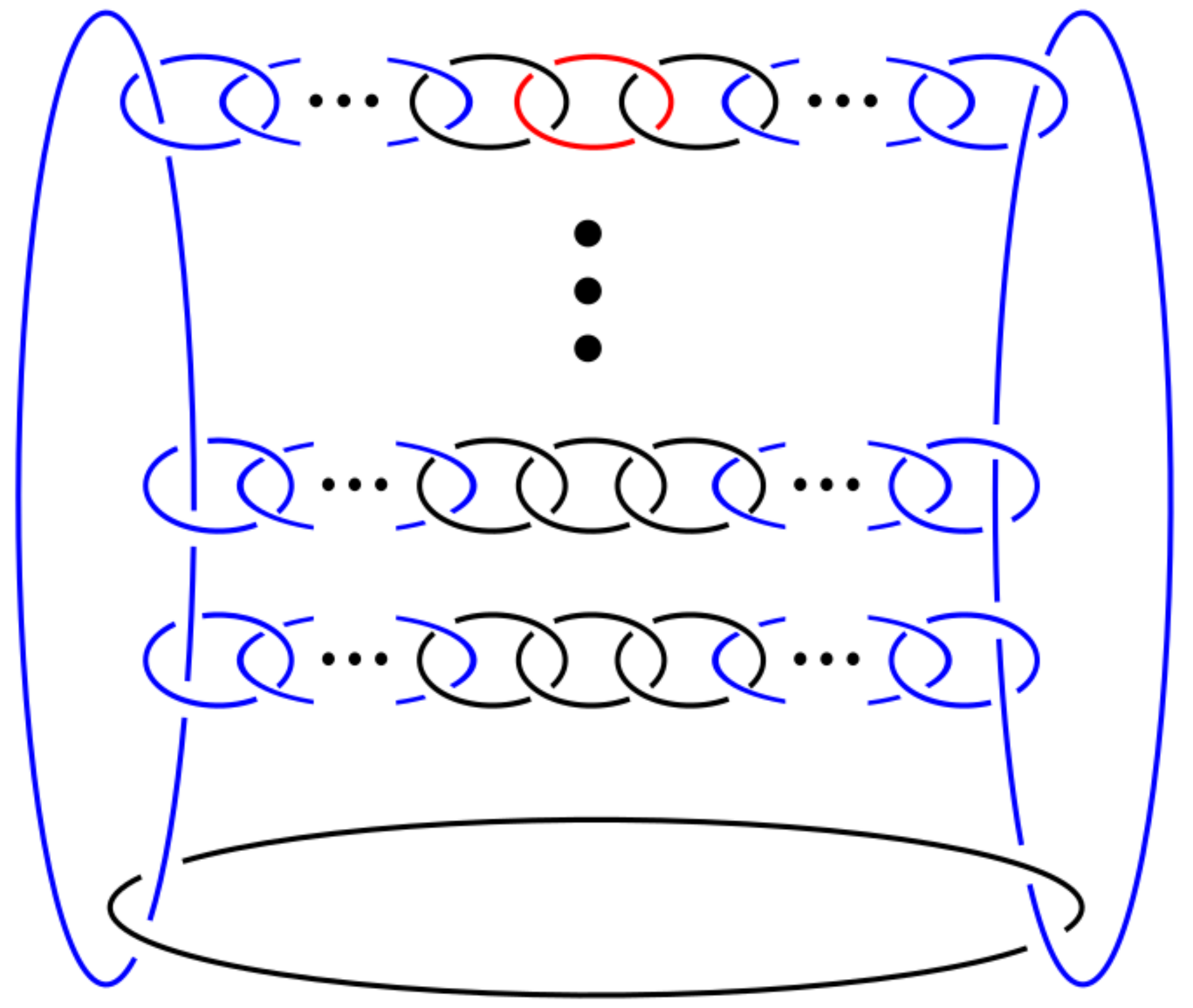}
\caption{ The link $D_{n,m}$.}
\label{dnmlink}
\end{figure}

As shown in Figure \ref{dnmlink},
the components of $F_{k,n,m}$ in the complement of $\mathcal{C}$, all of which lie in $D_{n,m}$ (appearing in blue), form a chain of shape $G$, where $G$ consists of two disjoint trees, which of course contain no cycles.  By Lemma \ref{chain lemma}, the manifold obtained by $(1,p)$-Dehn filling on the components not in $\mathcal{C}$ is a link complement.
\end{proof}

\section{Cusp Density of Knots}

\begin {definition}
Consider a link with two components, the first containing a tangle $T$ which does not create additional components, and the second a trivial component that wraps around two strands of the component containing $T$, as shown in Figure \ref{augmented}.
We call such a link an \emph{augmented tangle link}. When the tangle specifically has an X in it in the sense that the two strands labeled $a$ are connected to one another through the tangle, and the two strands labeled b are connected to one another through the tangle, then we say that the link is an \emph{augmented cross tangle link}.
\end{definition}

\begin{figure}[h]
\centering
 \includegraphics[scale =.5]{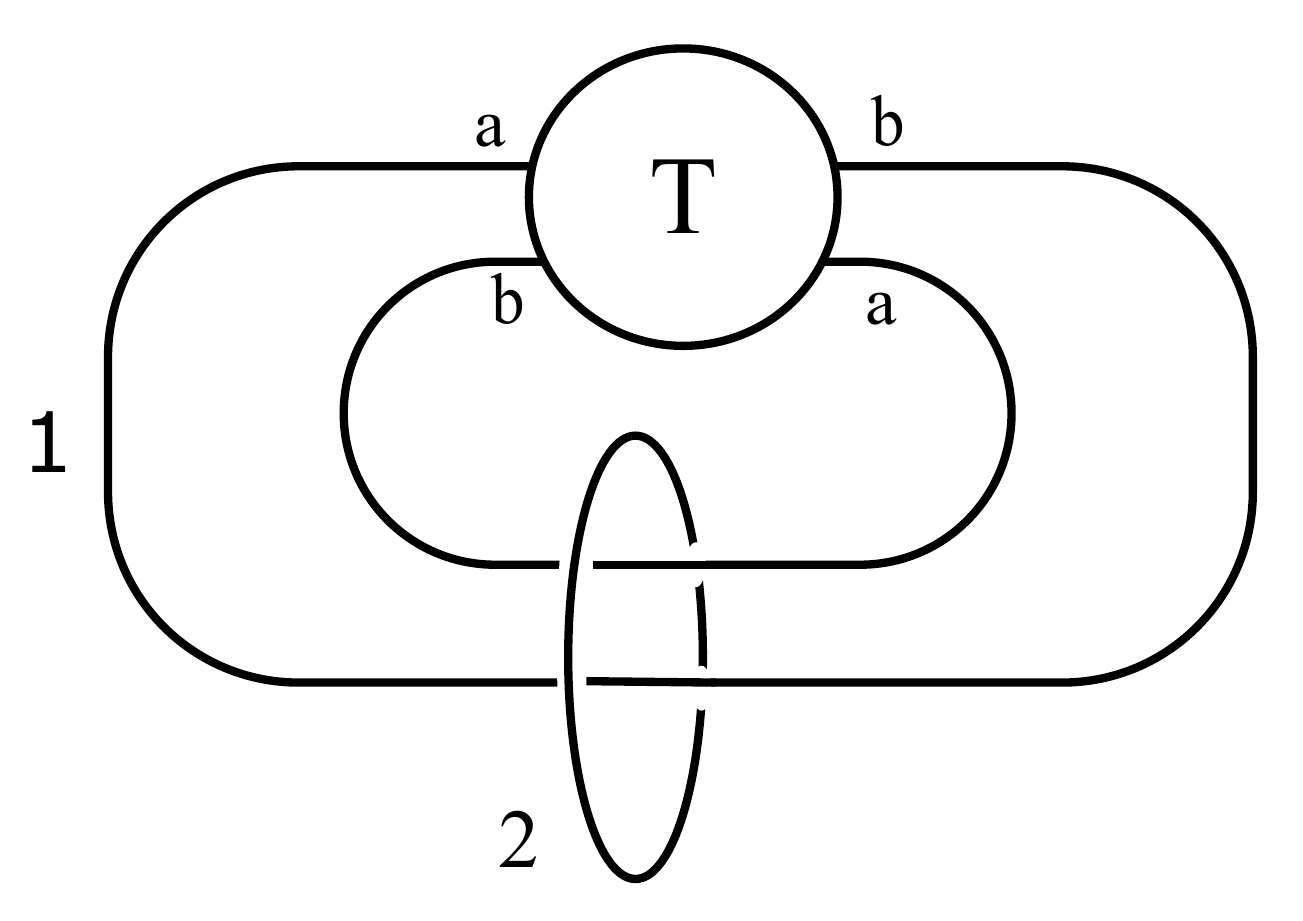}
\caption{ Augmented cross tangle link.}
\label{augmented}
\end{figure}

Note the twice-punctured disc bounded by cusp 2 in the complement of an augmented tangle link. A thrice-punctured sphere (or twice-punctured disk) is known to be totally geodesic and to have a unique hyperbolic structure (see for instance \cite{Adams7}).The disk appears in Figure \ref{twicepunc}(a) with certain edges marked. In Figure \ref{twicepunc}(b), we look at the cusp diagram, centering the horoball of cusp 2 at infinity. The following description of the twice punctured disc in the cusp diagram is crucial to our proofs. 

\begin{figure}
\centering
\begin{subfigure}{.47\textwidth}
  \centering
  \includegraphics[scale=.5]{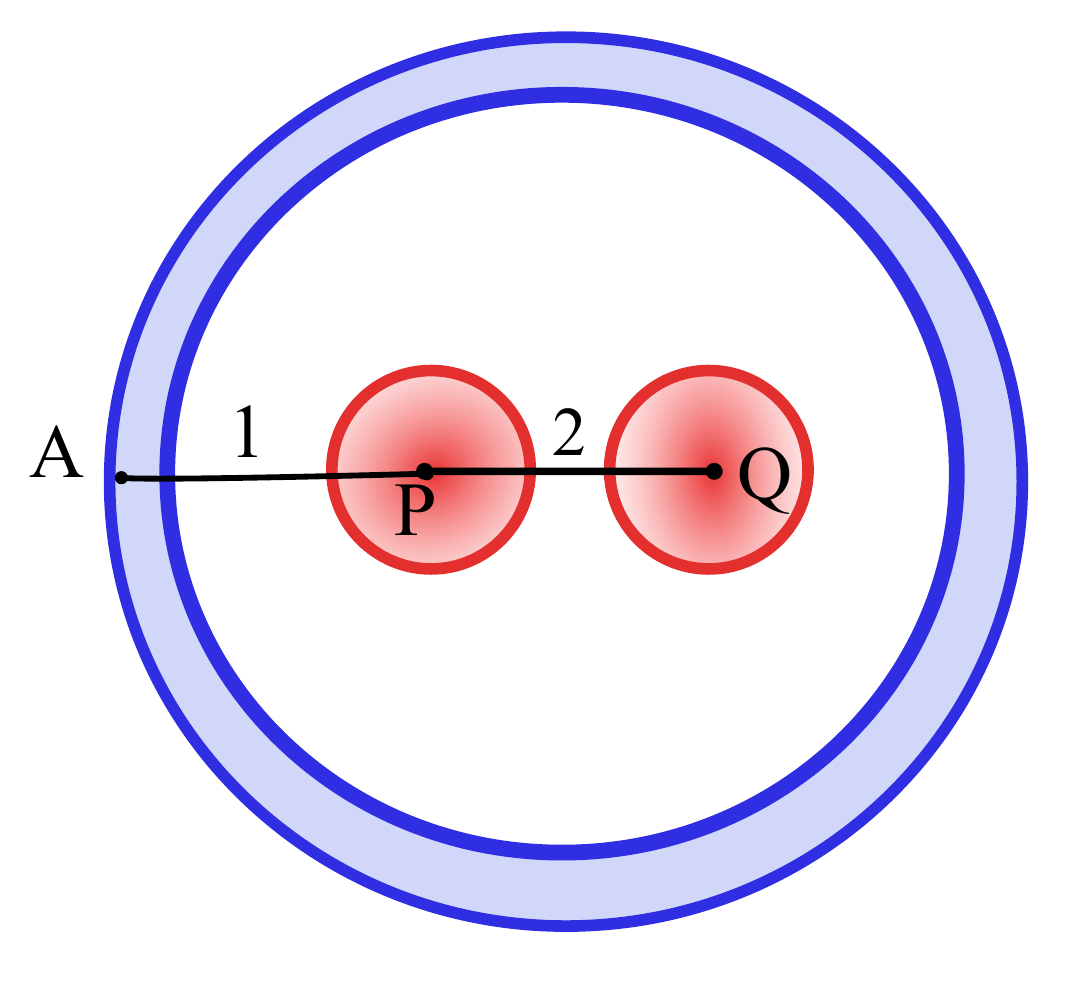}
  \caption{Cross section along cusp 2}
\end{subfigure}
\begin{subfigure}{.47\textwidth}
  \centering
  \includegraphics[scale = .4]{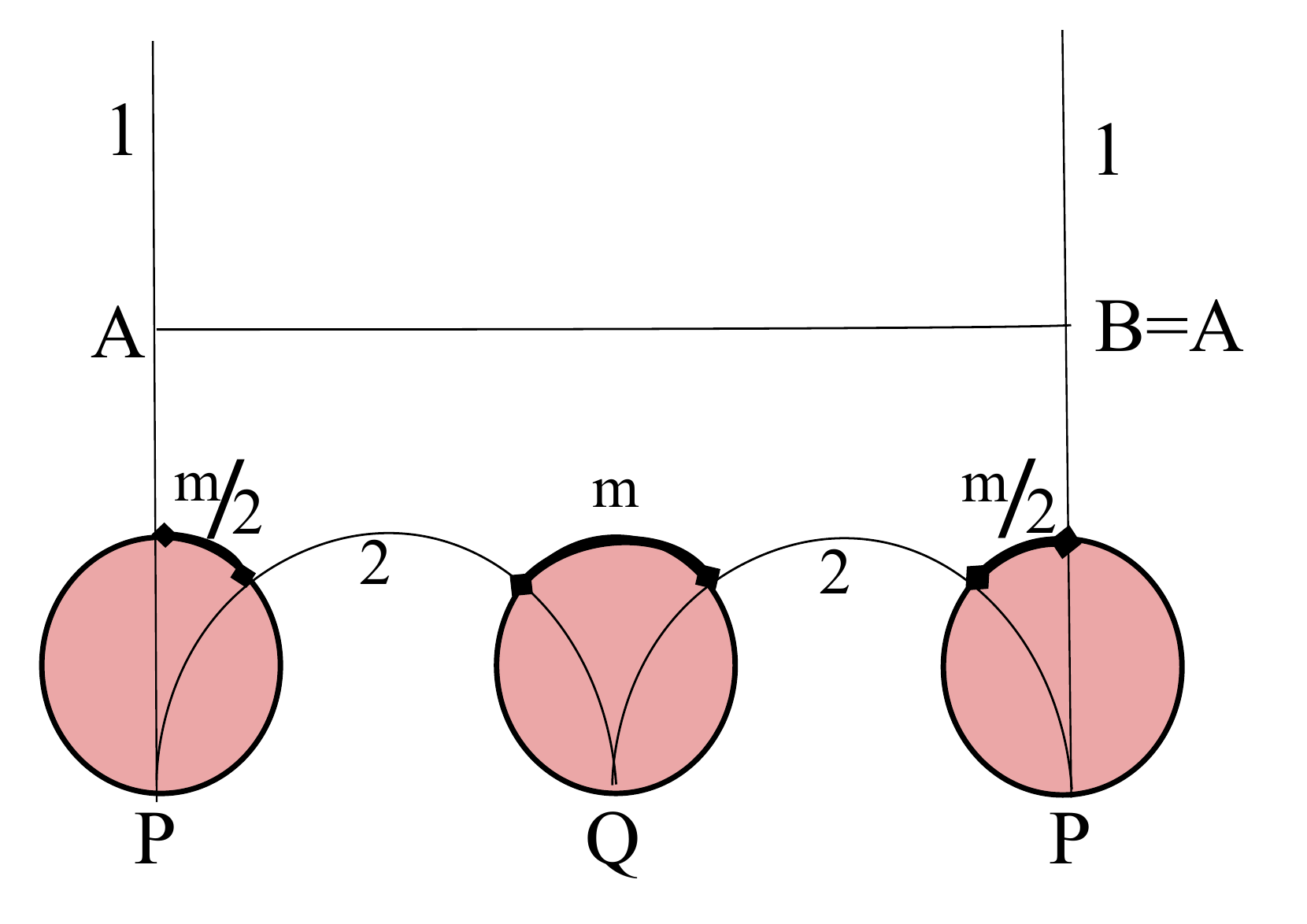}
\caption{Horoball Diagram}
\end{subfigure}
\caption{Twice punctured disc}
\label{twicepunc}
\end{figure}

We center the fundamental domain such that an endpoint of the longitude is directly above the  horoball of Cusp 1 that corresponds to the puncture $P$. There will be an identical ball on the `other' end since they are identified. We have a second puncture $Q$ in the twice-punctured disc and this would correspond to another horoball of Cusp 1 between the two horoballs at the ends of the longitude. This will be equidistant from the two  horoballs centered at $P$ since the geodesic (2) corresponding to $PQ$ should be the same when constructed from either ideal vertex corresponding to $P$. The edge 1 in Figure \ref{twicepunc}(b) appears as a geodesic connecting $P$ and infinity. 

On the leftmost horoball, the distance along the boundary cut off by the geodesics 1 and 2 is half the meridian (the other half is on the rightmost horoball). We also see the meridian as the distance cut off by the two geodesics labeled 2 along the boundary of the middle horoball. Since the meridian is constant along the whole length of cusp 1, all three horoballs are the same size. 

In order to find the interval in which cusp density is dense for knots we must first discuss poking and its implications for our constructions. 

\begin{definition} Given a thrice-punctured sphere $S$ in a hyperbolic 3-manifold $M$, we say that a cusp $C$ \emph{pokes} $S$ if, in the universal cover $\bbH^3$,  there is a horoball $H$ corresponding to $C$ and a geodesic plane $P$  that covers $S$ such that the center of $H$ is not on $\partial P$, but $H$ intersects $P$.  
See Figure \ref{poking}.

We say an augmented tangle link has poking if the cusp containing the tangle pokes  the twice-punctured disk corresponding to the trivial component.
\end{definition}

\begin{figure}{}
\begin{center}
\includegraphics[scale= .5]{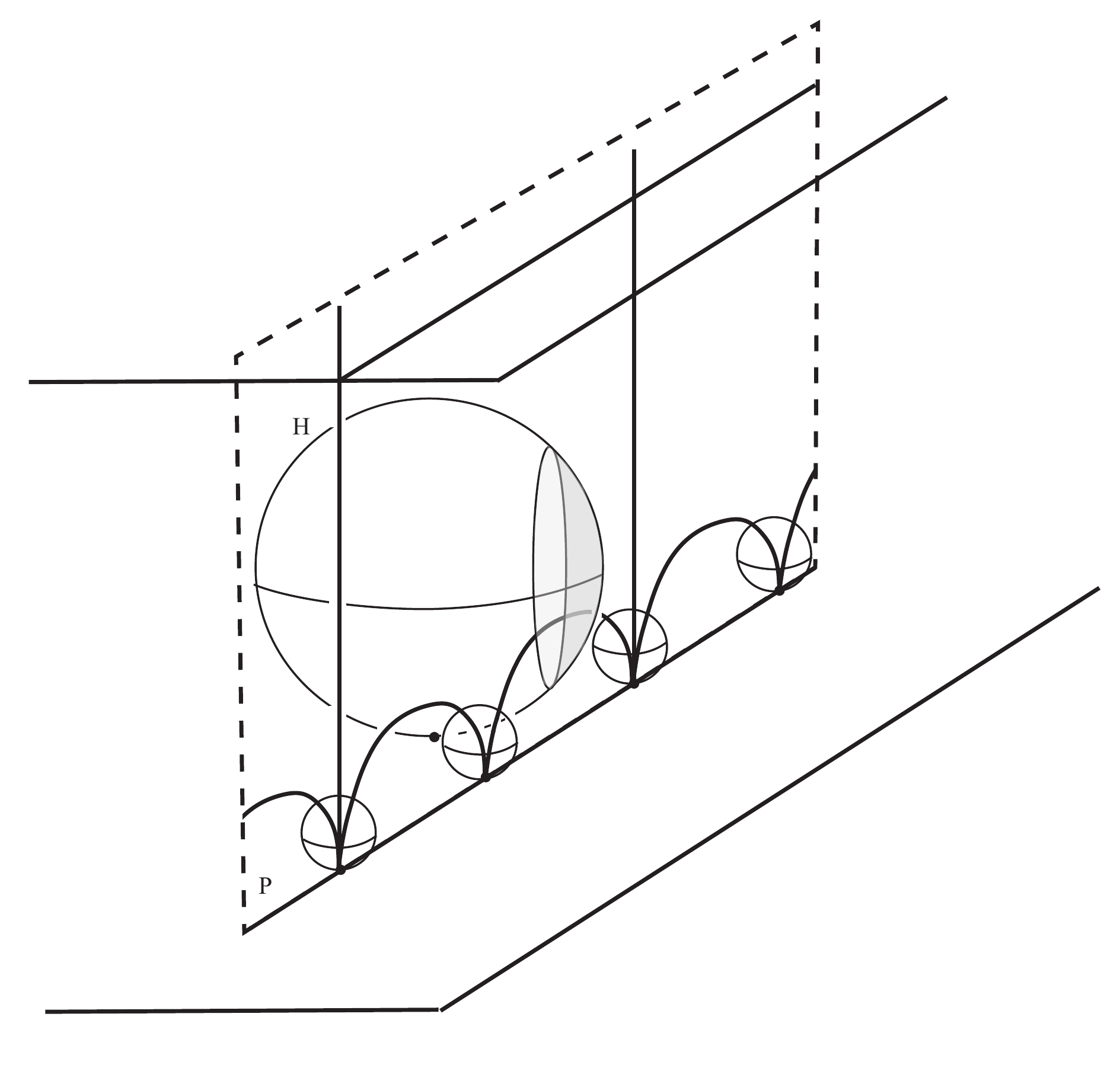}
\caption{Horoball Diagram: Poking across twice-punctured disk}
\label{poking}
\end{center}
\end{figure}

\begin{remark}
Poking has not been encountered in any augmented tangle link we have considered. We wonder whether poking ever occurs for augmented tangle links.  For the following constructions we assume its absence. 
\end{remark}


\begin{lemma} Let $L_1$ and $L_2$ be two augmented cross tangle links, both without poking, and let $L_3$  be their belted sum. Let $C_1$, $C_2$ and $C_3$ be the tangle components of each link.  Choose $m_1$, $m_2$ and $m_3$ to  be the meridian lengths of $C_1$, $C_2$ and $C_3$ when they are maximized first in their respective link complements and let $V_{C_1}$, $V_{C_2}$, and $V_{C_3}$ be their volumes.
Then  if  $m_1 \leq m_2$,  it must be that $m_3 = m_1$ and furthermore, $V_{C_3}=V_{C_1}+{\Big(\frac{m_1}{m_2}\Big) }^2 V_{C_2}$ .
\end{lemma}

\begin{proof}
As in \cite{Adams7}, a fundamental domain for the hyperbolic structure on $L_3$ in $\mathbb{H}^3$ is obtained by gluing together fundamental domains for $L_1$ and $L_2$ along  geodesic faces corresponding to the twice-punctured disks bounded by the trivial component in each. These disks are totally geodesic.  Let $P$ be a geodesic plane in $\mathbb{H}^3$ corresponding to he twice-punctured disks, once glued together.  The relative positions of the centers of horoballs in the fundamental domains for 
$L_1$ and $L_2$ are preserved. The cusp $C_3$ and the horoballs corresponding to it will come from the cusps $C_1$ and $C_2$ and the horoballs corresponding to them, but the sizes of the horoballs from $C_1$ and $C_2$ may need to be adjusted to match accordingly. 

Both $C_1$ and $C_2$ are maximal cusps in $L_1$ and $L_2$. Because there is no poking, there are no pairs of horoballs with centers on opposite sides of $P$ that touch one another. 

The centers of horoballs that are on the boundary of $P$ are now shared between the fundamental domains of $L_1$ and $L_2$, as they come from $C_1$ being glued to $C_2$ on the twice-punctured disk.
The cusps  are glued to each other in a way that takes a meridian on the twice punctured disk in $C_1$ to a meridian on the twice punctured disk of $C_2$. Note that neither $m_1$ nor $m_2$ can be made any larger when they appear in $L_3$ since horoballs that do not poke across the twice-punctured disk touch one another in the fundamental domains for $L_1$ and $L_2$ that we are using to construct the fundamental domain for $L_3$. It follows that the meridian $m_3$ of $C_3$ cannot be larger than either $m_1$ or $m_2$. As we assumed that $m_1 \leq m_2$,we must shrink the cusp $C_2$ back until its meridian matches with $m_1$, and this means that $m_3 = m_1$. 

This  implies that the part of $C_3$ that corresponds to $C_2$ from before is the same except with the meridian scaled down by a factor of $\frac{m_1}{m_2}$. Since the cusp shape is similar to before, the longitude is also scaled down by a factor of $\frac{m_1}{m_2}$. Therefore, as the cusp volume is directly proportional to cusp area, the new cusp volume corresponding to the $C_2$ part of $C_3$ is scaled down by $\big(\frac{m_1}{m_2}\big)^2$ and so is equal to $\big(\frac{m_1}{m_2}\big)^2V_{C_2}$. As the meridian of $C_1$ is unchanged under the gluing, the cusp volume of the part of $C_3$ corresponding to $C_1$ remains the same. Therefore, as $C_3$ is just the union of the parts corresponding to $C_1$ and $C_2$, we have that $V_{C_3} = V_{C_1} + \big(\frac{m_1}{m_2}\big)^2V_{C_2}$.  
\end{proof}

\begin{lemma}
\label{knot density}
Let $L_1$ and $L_2$ be  augmented cross tangle links with no poking and tangle components $C_1$ and $C_2$ of respective restricted cusp density $cd_1$ and $cd_2$ and meridians of lengths $m_1$ and $m_2$ where $m_1\leq m_2$. Then the cusp densities of knot complements are dense in $[cd_1,{\big(\frac{m_1}{m_2}\big)}^2cd_2]$ if $cd_1 \leq {\big(\frac{m_1}{m_2}\big)}^2cd_2$, or in $[{\big(\frac{m_1}{m_2}\big)}^2cd_2,cd_1]$ if ${\big(\frac{m_1}{m_2}\big)}^2cd_2 < cd_1$.
\end{lemma}

\begin{proof}
Let $L_{k,p}$ be the link that results from taking a belted sum of $k$ copies of $L_1$ and $p$ copies of $L_2$. Let $C_{k,p}$ be its tangle component and let $cd_{C_{k,p}}$ be its restricted cusp density. Note that when $k + p$ is odd then $L_{k,p}$ is a two component link where one of the components bounds the  twice-punctured disk.Then there exists a knot with cusp density arbitrarily close to $cd_{C_{k,p}}$ that is obtained by doing high $(1,p)$-Dehn filling on the trivial component of $L_{k,p}$. 

We will assume to begin that $cd_1 \leq {\big(\frac{m_1}{m_2}\big)}^2cd_2$. Now, for $k,p \geq 1$,  $L_{k,p}$ will have $m_{k,p} = m_1$ since $m_1$ is the smallest of all the meridians of the belted sum components. Therefore, all $p$ of the  $L_2$ components of the belted sum will add volume ${\big(\frac{m_1}{m_2}\big)}^2V_{C_2}$ to $V_{C_{k,p}}$. Additionally, all $k$ copies of $L_1$ will add volume $V_{C_1}$ to $V_{C_{k,p}}$. Thus,  $V_{C_{k,p}}= kV_{C_1} + p {\big(\frac{m_1}{m_2}\big)}^2V_{C_2}$. As volumes add under belted sums, the restricted cusp density is given by: 
\[cd_{C_{k,p}} = \frac{kV_{C_1} + p {\big(\frac{m_1}{m_2}\big)}^2V_{C_2}}{k \cdot vol(L_1) + p \cdot vol(L_2)}.\] 

Divide by $k$ in the numerator and denominator, and let $t = \frac{p}{k}$ to obtain: \[ f(t)= \frac{V_{C_1} + t {\big(\frac{m_1}{m_2}\big)}^2V_{C_2}}{vol(L_1) + t \cdot vol(L_2)}.\]
Then $\lim\limits_{t\to\infty} f(t) = \frac{\big(\frac{m_1}{m_2}\big)^2V_{C_1}}{vol(L_2)} = \big(\frac{m_1}{m_2}\big)^2 cd_2$, and $\lim\limits_{t\to 0} f(t)= \frac{V_{C_1}}{vol(L_1)} = cd_1$. Additionally as $f(t)$ is a continuous function of $t$, it takes on all values in $[c_1,{\big(\frac{m_1}{m_2}\big)}^2c_2]$ . As $k$ and $p$ are integers, for any value of $t$ we can find a $k$ and a $p$ such that $\frac{k}{p}$ is arbitrarily close to $t$, and furthermore we can do so such that $k+p$ is odd. As a knot can approach any of these values, the cusp densities of knots are dense in $[c_1,{\big(\frac{m_1}{m_2}\big)}^2c_2]$.

When ${\big(\frac{m_1}{m_2}\big)}^2c_2 < c_1$, a similar argument holds to show knot cusp densities are dense in the interval $[{\big(\frac{m_1}{m_2}\big)}^2c_2,c_1]$.
\end{proof}

Now we would like to find two links satisfying the conditions of Lemma \ref{knot density} that give the largest interval in which cusp densities are dense. Just as good for our purposes, we will instead find two families of augmented cross tangle links with no poking that each have meridians of the tangle component approaching 2, and have restricted cusp densities approaching 0 and $.6826\dots$ respectively. 

In order to describe the family of augmented cross tangles with no poking that have restricted cusp density approaching 0, we utilize the alternating daisy chains. Denote such a chain with $n$ components $D_n$. Let $M_n$ denote a choice of a single cusp, individually maximized in $D_n$. As shown in \cite{Adams6}, as $n$ approaches infinity, $M_n$ approaches the maximized cusp of the Borromean rings in structure. This means that the meridian $m_n$ of $M_n$ approaches 2 and the restricted maximized cusp volume $cv_R(D_n,M_n)$ approaches 4. Additionally, as $n$ approaches infinity, $vol(D_n)$ approaches infinity, and so $cd_R(D_n,M_n)$ approaches 0. Additionally, as $M_n$ approaches in structure the Borromean rings maximized cusp, and as the Borromean rings have no poking, for sufficiently large $n$, neither will $M_n$. 

Now, let $D_{n,p}$ be the manifold that comes from performing $(1,p)$-Dehn filling on every component of $D_n$ except for $M_n$ and the two components of $D_n$ that are adjacent to and linked with $M_n$. Let $M_{n,p}$ be the  component from the new manifold corresponding to $M_n$. As $p$ approaches infinity, $D_{n,p}$ will approach $D_n$ in volume, and the three cusps $M_{n,p}$ and the two surrounding cusps will approach their previous volumes and structures before filling. By Lemma $\ref{chain lemma}$, we know that $D_{n,p}$ is a link complement. Specifically $D_{n,p}$ will have three cusps and since the component $M_{n,p}$ is not Dehn filled and neither are the surrounding two components, $M_{n,p}$ will remain a trivial component that bounds a twice-punctured disk. 

Now for sufficiently large $n$ and $p$, $M_{n,p}$ will have no poking as it approaches in structure $M_n$, and therefore we can insert a half-twist through the twice-punctured disk bounded by $M_{n,p}$ that connects the other two cusps, but otherwise does not change any cusp structure but to unify the previous two cusps assuming we are in the case with no poking. Call this new two-component link $L_{n,p}$, and note that now $L_{n,p}$ is in augmented cross tangle form. Additionally, it is still maintained that as $n$ and $p$ approach infinity, the restricted cusp density of $L_{n,p}$ for both components approaches 0, and additionally, for sufficiently high $n$ and $p$, $L_{n,p}$ has no poking. Discard all $L_{n,p}$ that possibly have poking, and call this remaining family of augmented cross tangles with restricted cusp density approaching 0 and no poking $F_1$.

To construct the family with high cusp density, start with the maximally twisted four chain. The complement of the maximally twisted four chain is composed of ten ideal regular tetrahedron and so it has volume $10.149\dots = 10v_{tet}$. Depending on the order of maximization, opposite cusps have volumes of $2\sqrt{3}$ and $\frac{\sqrt{3}}{2}$ respectively. Additionally the meridian of the first maximized components is exactly 2. Then, insert a half-twist through one of the twice punctured disks in order to connect the two cusps that each have a volume of $2\sqrt{3}$. The resulting link complement will still be made out of ten ideal regular tetrahedron and the cusps will be unchanged except that the two previous opposite cusps are joined so that they are a single cusp with volume $2\sqrt{3} + 2\sqrt{3} = 4\sqrt{3}$ and this larger component still has a meridian of 2. Therefore the restricted cusp density of the largest component after inserting a half-twist is equal to $\frac{4\sqrt{3}}{10v_{tet}} = .6826\dots$. 

Now perform $(1,p)$-Dehn filling on the trivial component opposite the one we put a half-twist through in order to obtain links that have restricted cusp densities approaching $.6826\dots$, and meridians of the tangle component approaching 2. Note that all such Dehn fillings are exactly the family of augmented cross tangle links that we desire. Additionally, for a high enough $p$, all links in the family that come from $(1,p)$-Dehn filling will have no poking as they approach having a perfect wall of full-sized balls along the twice-punctured disk component because they approach a manifold made only from ideal regular tetrahedra with the maximum ball packing allowed in hyperbolic space. Call this family of links $F_2$.

\begin{theorem}
Cusp densities of knots are dense in $[0,0.6826\dots]$.
\end{theorem}

\begin{proof}
There exist links $l_1\in F_1$ and $l_2\in F_2$ arbitrarily close in cusp density to 0 and $.6826\dots$ respectively, and both families consist of links that satisfy the conditions of Lemma $\ref{knot density}$. Therefore cusp densities of knots are dense in $[0,.6826\dots]$. 
\end{proof}


\section{Cusp Crossing Density}

In this section we consider cusp crossing density of hyperbolic links, given as the ratio of maximum cusp volume to crossing number. 

\begin {lemma} The cusp crossing densities of hyperbolic knots and links are bounded above by $\frac{\sqrt{3} v_{oct}}{2 v_{tet}} \approx 3.1263\dots$ and there exist families of knots with cusp crossing density approaching $0$. 
\end{lemma}

\begin{proof}
By definition $0 < d_{cc}(L)$. We find a family of knots with cusp crossing density approaching $0$. Consider the Whitehead link. Performing $(1,q)$-Dehn filling on one component results in the twist knot family, $Tw_q$, where crossing number is $2q + 2$. By Thurston's Hyperbolic Dehn Filling Theorem (cf. \cite{Thurston}), as $q$ becomes sufficiently large, the cusp volumes of the twist knots approach that of one component of the Whitehead link. Thus $\lim\limits_{q \rightarrow \infty} d_{cc}(Tw_q) = \lim\limits_{q \rightarrow \infty} \frac{cv(Tw_q)}{c(Tw_q)}= 0$. 

\begin{figure}[h]
\begin{center}
\includegraphics[scale=.5]{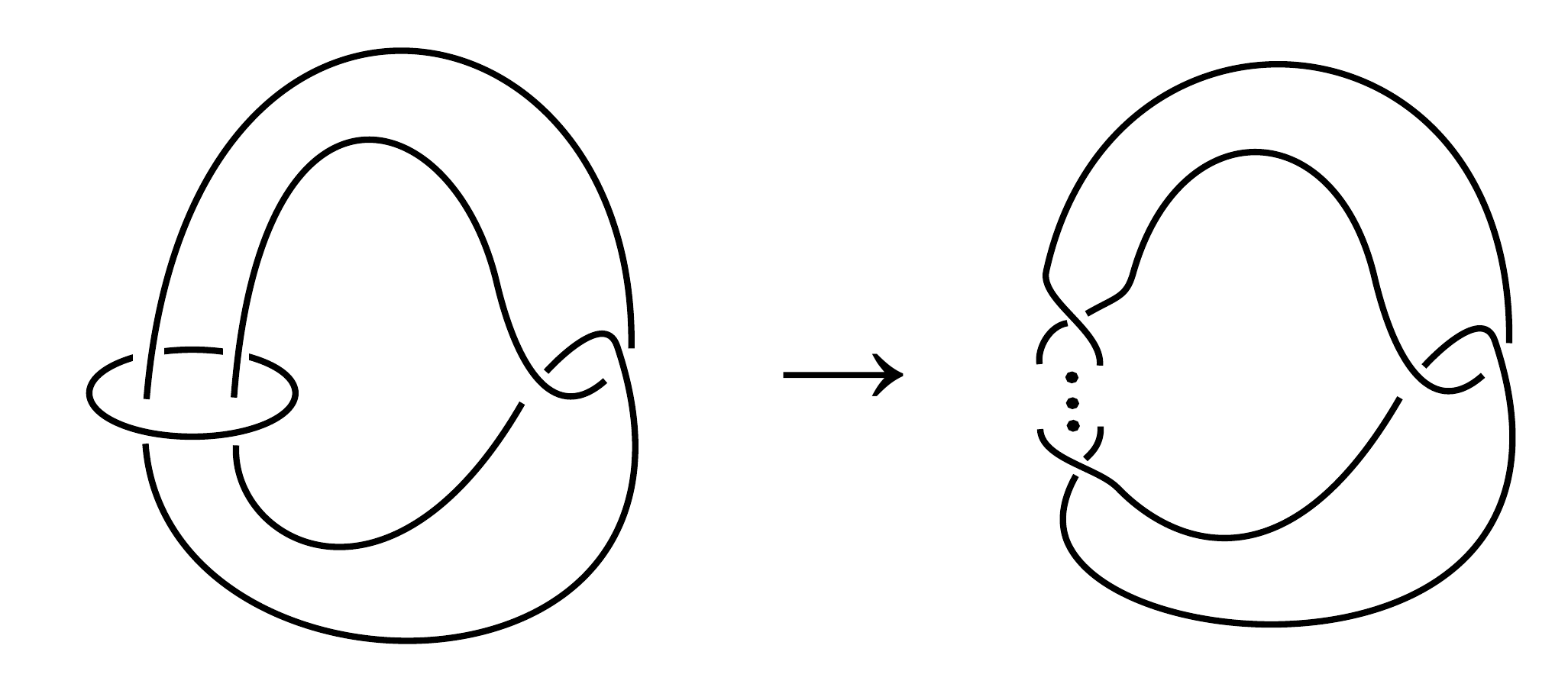}
\caption{The Whitehead link Dehn fills to become the twist knot family.}
\end{center}
\end{figure}

We obtain our upper bound from previous results about packing and cusp density.

 Recall that the densest packing of horoballs in hyperbolic space is $\frac{\sqrt{3}}{2v_{tet}} \approx .853$. Thus $\frac{cv(L)}{vol(L)}\leq .853\dots$.  Since volume density is bounded above by $v_{oct}$ (which follows from a result of D. 
 Thurston, see \cite{Adams3} for an explanation), $\frac{vol(L)}{c(L)} \leq v_{oct}$. Therefore, for any hyperbolic link $L$, \[d_{cc}(L)= \frac{cv(L) v_{oct}}{c(L)v_{oct}}\leq \frac{cv(L)}{vol(L)}v_{oct}\leq (.853\dots)(v
_{oct})\approx 3.1263.\]
\end{proof}

\begin{lemma}There exists families of links with cusp crossing density approaching 3 from below.
\end{lemma}

\begin{figure}{}
\begin{center}
\includegraphics[scale= .6]{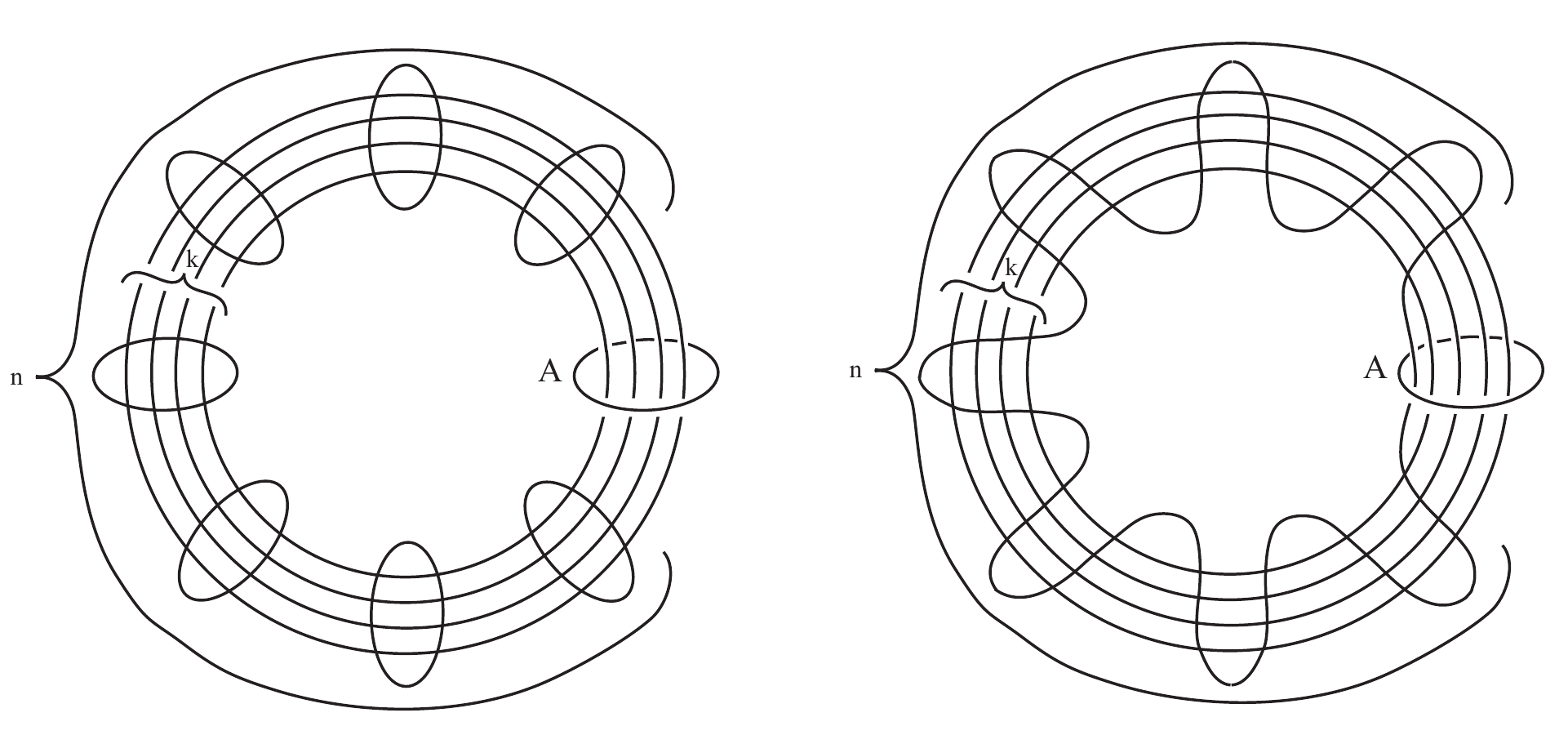}
\caption{Link families with cusp crossing density approaching 3 from below.}
\label{octahedral1}
\end{center}
\end{figure}

\begin{proof} Consider the two link families $L(n,k)$ and $L'(n,k)$ appearing in Figure \ref{octahedral1}. Ignoring component $A$, the link diagrams are to be alternating. If not for the presence of the component labelled $A$, the link complements $L(n,k)$ and and $L'(n,k)$ could be decomposed, using face-centered bipyramids as in \cite{Adams3} into $2n(k-1)$ octahedra and two $2n$-bipyramids for the innermost and outermost faces. All of these bipyramids have two ideal vertices at top and bottom while all the vertices around their equators are finite. The finite vertices are identified into two vertex classes, one we call $U$, which is the up vertex above the projection plane and one we call $D$, which is the down vertex below the projection plane. At crossings that do not involve bigons, there are three edge classes, one from $D$ to the bottom of the lower strand, one from the top of the lower strand to the bottom of the upper strand, and one from the top of the upper strand to $U$. 

For crossings that do not touch a bigon, there are four edges in each of the three edge class corresponding to the crossing, one on each of the four octahedra meeting at this crossing. 

We now consider crossings that do touch a bigon. The edges spanning the two crossings on the bigon are isotopic to one another and so they occur in the same edge class. This edge class appears twice on adjacent edges on the equator of the octahedron coming from the square that shares an edge with the bigon, and once each on the adjacent octahedra (appearing as edge class 1 in Figure \ref{edgeclasses}). However, there are currently other edge classes that do not contain four edges. 

As in \cite{Adams3}, the addition of the $A$ component, which passes through both $D$ and $U$, skewers  the inner and outer  bipyramids and collapses them down to polygons. It also makes all of the finite vertices of the octahedra into ideal vertices. The identification of the edges on the top of the two collapsed bipyramids with the edges on the bottom (edge class 4 with 6, 5 with 2 and 3 with 7 in Figure \ref{edgeclasses}) ensures that all edge classes now consist of four edges each. For instance, after these identifications, edge class 2 appears twice on the middle octahedron and once each on the other two octahedra.  Edge class 4 will appear twice on the upper left octahedron, once on the middle octahedron and once more on an octahedron further to the upper left and not appearing in this figure. 

\begin{figure}{}
\begin{center}
\includegraphics[scale= .8]{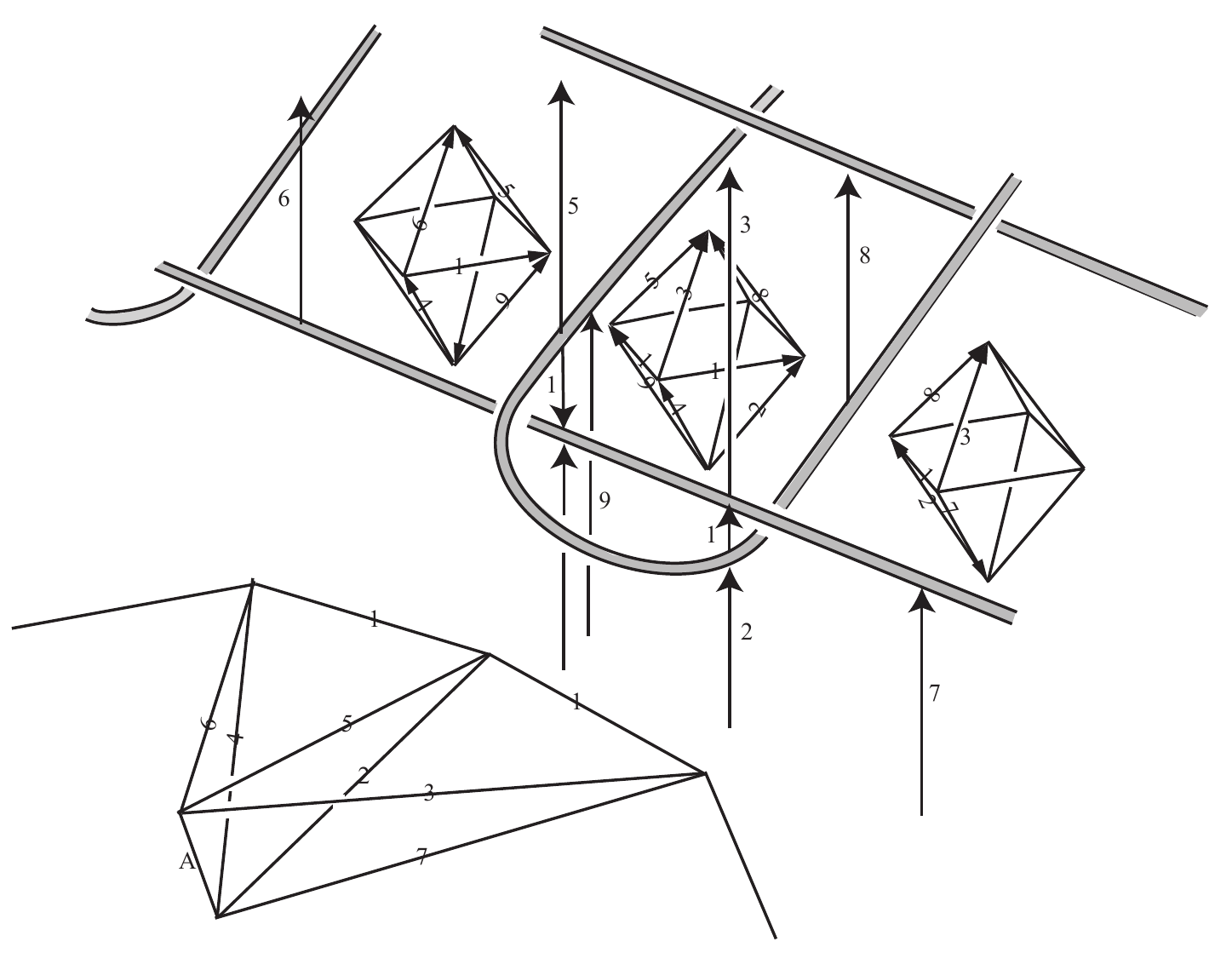}
\caption{Component $A$ collapses the $2n$-bipyramids and identifies edge classes to have four edges in each.}
\label{edgeclasses}
\end{center}
\end{figure}

 Thus, by taking all ideal regular octahedra, all of which have all dihedral angles of $\pi/2$,  we satisfy all of the gluing equations of Thurston including the completeness equations, since the boundary of the cusps are made up of equal sized squares. We therefore obtain the unique hyperbolic structure on the complement, with a volume of $2n(k-1) v_{oct}$.

If we choose the \emph{standard packing} of horoballs for this collection of octahedra, meaning the three horoballs that correspond to the vertices of a triangular face are all pairwise tangent, then they together form a valid set of cusps since the gluings on the faces respect this packing.  (See Figure \ref{squarehoroballs}.) They produce a volume of 3 in each octahedron, 1/2 from each vertex. If the maximal cusp volume of the manifold corresponds to this packing, then the maximal cusp volume will be $3(2n(k-1)) = 6n(k-1)$.

\begin{figure}[h]
\begin{center}
\includegraphics[scale =.5]{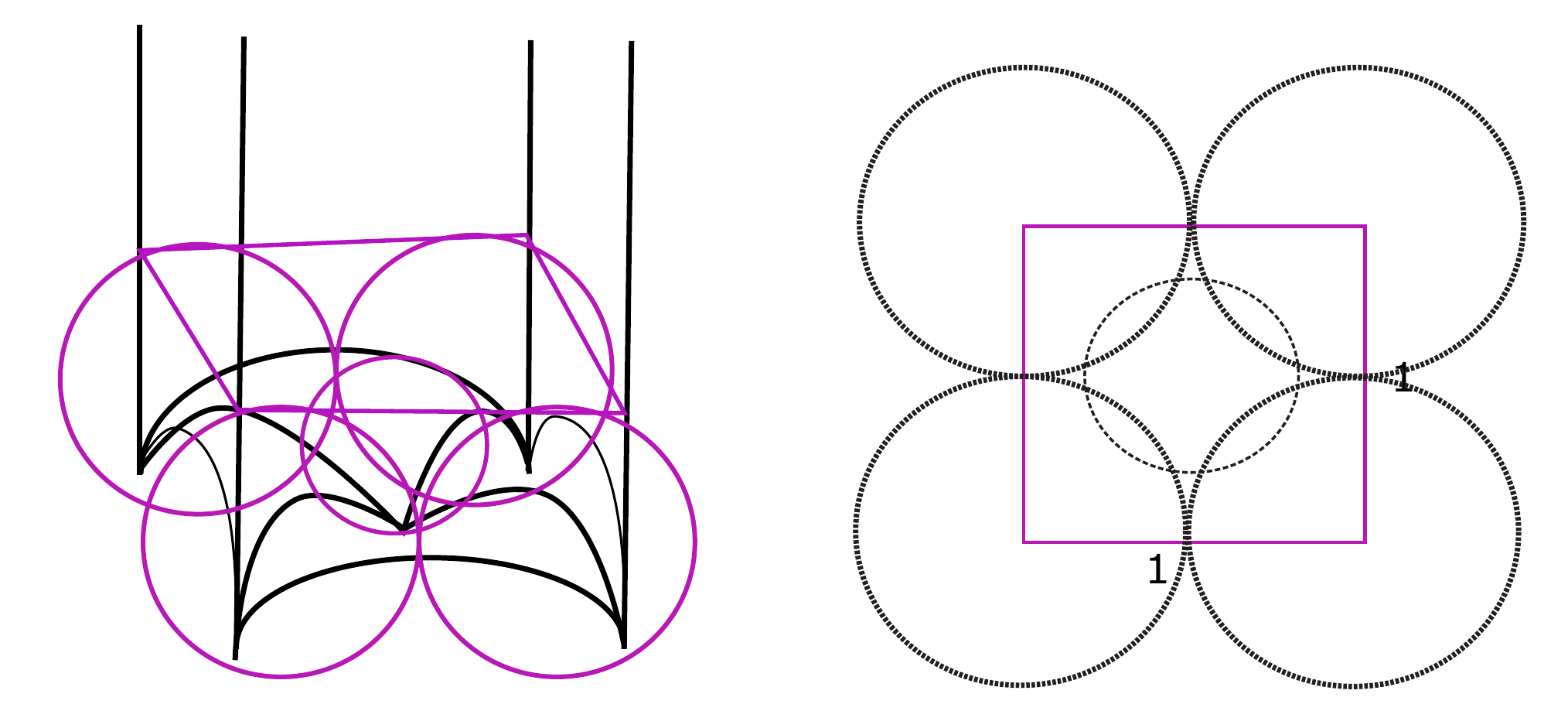}
\caption{Standard packing of horoballs for an ideal regular octahedron seen from the side and from above.}\label{squarehoroballs}
\end{center}
\end{figure}

Momentarily ignoring component $A$, the crossing number for such a link  must be $2nk$ since this is an alternating link. Since the component $A$ must link each of $k$ or $k+1$ of the components, we have that $c(L(n,k)) = 2nk + 2k$ and $c(L'(n,k)) = 2nk + 2k + 2$. 

Therefore: \[d_{cc}(L(n,k)) = \frac{6n(k-1)}{2nk + 2k} \hspace{.2cm} \text{and} \hspace{.2cm} d_{cc}(L'(n,k)) = \frac{6n(k-1)}{2nk + 2k + 2}.\]

As $k$ and $n$ approach $\infty$, these both approach 3 from below, as we wanted to show.

So, to complete the proof, we must show that there is no other choice of cusps that could generate a higher cusp volume than the choice that yields the standard packing. Any other choice can be obtained by increasing the size of certain cusps and decreasing sizes of cusps that touch them, starting from the standard packing. Note that if there are $n$ cusps, a maximal cusp volume must achieve at least $n$ points of tangency (cf.\cite{PT}).In order to obtain a higher volume in the set of cusps, at least one of the cusps must be expanded. There are three cases.
   
   Suppose a cusp has two horoballs with centers that are adjacent  vertices on one of the octahedra. Then it cannot be expanded as it is already touching itself in the standard packing. Suppose a cusp has two horoballs in one octahedron such that they are at opposite vertices. Then expanding this cusp forces all of the cusps corresponding to the other vertices to shrink and it shrinks total volume until the horoballs at the opposite vertices touch at the center of the octahedron. There the volume in the octahedron is 1 for each of these horoballs, while the remaining four horoballs have been shrunk down to size 1/4. However, the total cusp volume has bottomed out and come up again to  3.  See p.13-15 of \cite{KS} for details.
   
   In the last case, there is a cusp $C$ such that in each octahedron, at most one vertex corresponds to it.   However, in the case of  our link complements, if we take one of the octahedra $H_1$ to have vertex at $\infty$ corresponding to $C$, there is a second octahedron glued to one of the bottom faces of the octahedron, placing a $C$ vertex directly beneath the center of a vertical face of $H_1$. In other words, there are two faces on two octahedra such that those faces share an edge, they are in a totally geodesic plane and such that their opposite vertices both correspond to $C$. 
   
   So when we expand $C$, the horoballs at these vertices will bump into one another along that shared edge. This will prevent the horizontal horosphere corresponding to the $C$ vertex at $\infty$ from expanding down past the equatorial edges of $H_1$. In fact, when it reaches those edges, it will be tangent to all of them. Since these edges peak at a height of 1/2, the volume in this horoball in $H_1$ will be 2. Each of the four vertices on the equator will have horoballs with volume 1/8, and the horoball at the center vertex in the boundary plane will have volume 1/2. So once again, the cusp volume in this octahedron is at most 3. So there is no gain in total cusp volume. See p. 13-15 of \cite{KS} for details in this case as well. 
   
   Thus, in all cases, there cannot be a choice of cusps that generates more total volume in the cusps than the standard packing.

\end{proof}



\begin{remark}
A family of knots with notably high cusp crossing densities is the weaving knots described by Champanerkar, Kofman, and Purcell\cite{CKP1}. They define the weaving knots as alternating versions of the standard projections of the  $(p,q)$-torus knot or link and denote them $W(p,q)$. Up through knots of 14 crossings, they have the highest cusp crossing density. These knots are useful because, as described in \cite{CKP2},  as $p$ and $q$ go to infinity, their volume densities limit to the volume density of the infinite square weave, which is the maximum possible, $v_{oct} = 3.6638\dots$. Utilizing  the D. Thurston decomposition of  link complements via one octahedron at each crossing, in the case of the weaving knots, these octahedra limit to regular octahedra as the weaving knots grow towards the infinite square weave. This raises volume density. Since $d_{cc}(K) \leq d_{vol}(K) v_{oct}$, this suggests the possibility of higher cusp crossing density values. Figure \ref{w(10,11)} shows weaving knot $W(10,11)$. Although the largest weaving knot that we have input that does not crash SnapPy is $W(10,11)$, one can do the following to get higher cusp crossing densities for knots. Take weaving knot $W(p,q)$, and  construct a 2-component link $W'(p,q)$ with first component $C_1$ corresponding to $W(p,q)$ and second component $C_2$ obtained by adding a trivial component   through the center hole linking around the entire projection. Then the cusp crossing density of $C_1$ in $W'(p,q)$ is equal to the limit of the cusp crossing density of $W(n,q)$ as $n \to \infty$. Moreover, since $W'(p,q)$ is a $p$-fold cover of the link $W'(1,q)$, its first cusp has the same cusp crossing density as does the first cusp of $W'(p,q)$. Hence, we can find the cusp crossing density of $W('1,q)$, and there will be knots of the form $W(n,q)$ with cusp crossing densities approaching this number. Using SnapPy, we have found that $d_{cc}(W(n,26))$ approaches 1.706 as $n\to \infty$.  So there do exist knots with cusp crossing density approaching this number.
\end{remark}

\begin{question} Do the weaving knots have the highest cusp crossing density of all knots, i.e. for any knot, is there a weaving knot that has a higher cusp crossing density?
\end{question}

\begin{figure}[h]
\begin{center}
\includegraphics[scale = .4]{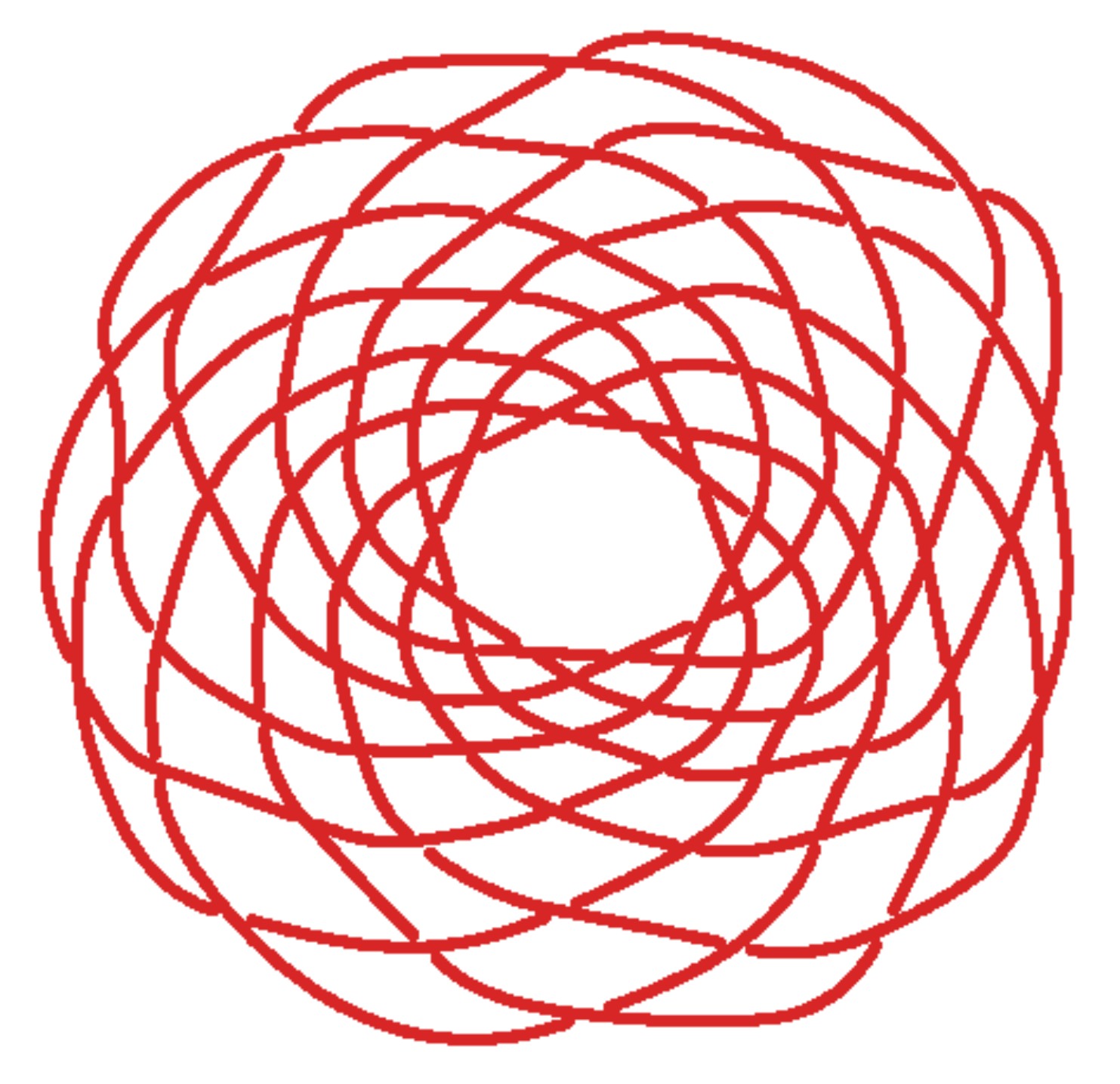}
\caption{Weaving knot $W(10,11)$ (SnapPy)}
\label{w(10,11)}
\end{center}
\end{figure}

Consider an augmented cross tangle link as in Figure 4, such that the tangle is alternating. Let $d=\frac{V_1+V_2}{c-4}$, where $V_1$ is the restricted cusp volume of the cusp containing the tangle, which has been maximized first, and $V_2$ is the cusp volume of the cusp containing the trivial component maximized relative to the cusp of the component containing $T$, and $c$ is the crossing number of $L$.

\begin{lemma}\label{twocomponent}
Cusp crossing densities of two component links are dense in the interval $[0, d]$.  
\end{lemma}


\begin{figure}[h]
\begin{center}
\includegraphics[scale = .5]{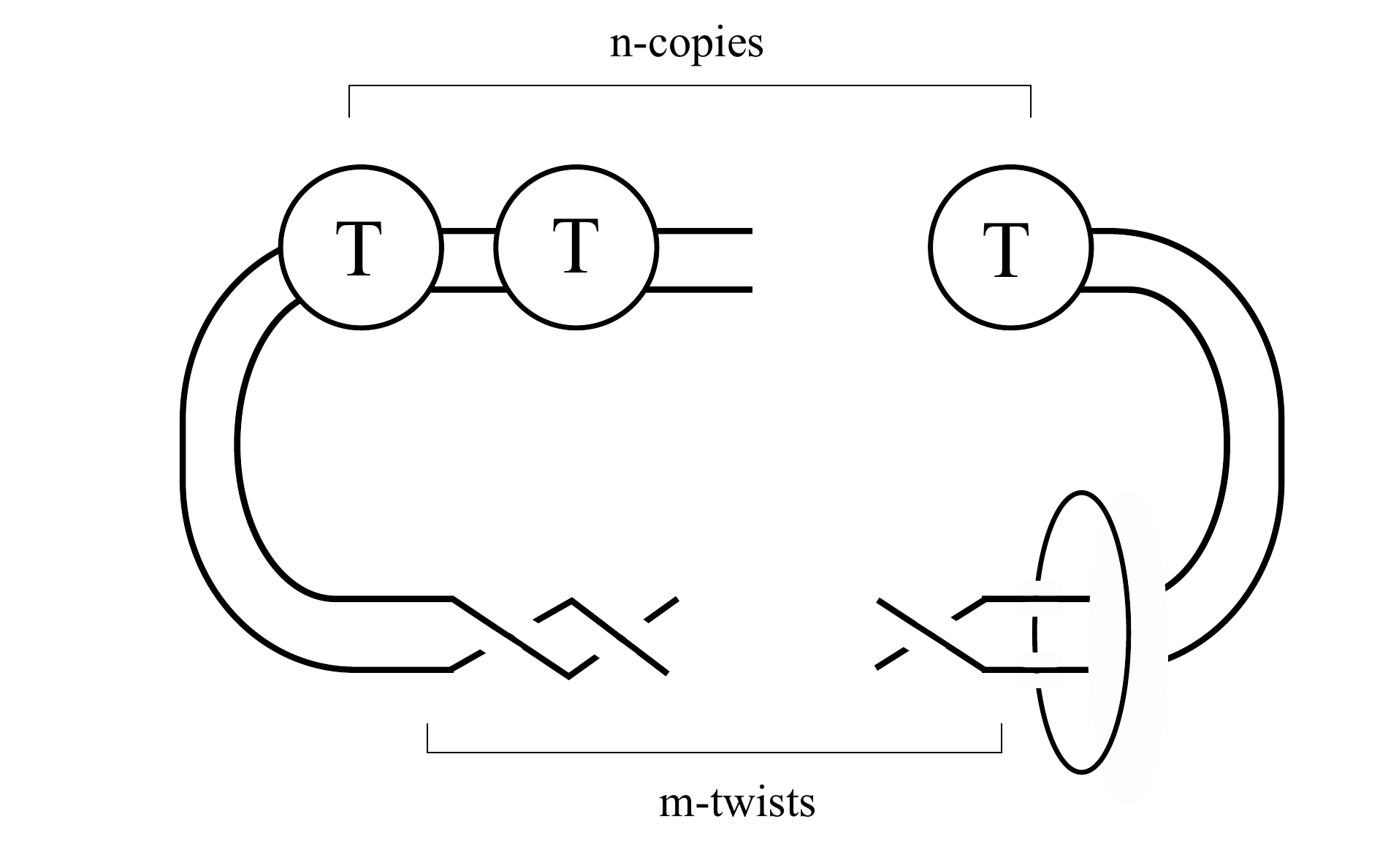}
\caption{n-fold cover followed by (1,m) surgery}
\end{center}
\end{figure}

\begin{proof} 
We take an $n$-fold cyclic cover (where $n$ is odd) with respect to the trivial component. Since $n$ is odd, the resulting link is still an augmented knot. The initial cusp volume $V$ is the sum of $V_1$ and $V_2$. The n-fold cover results in a link with cusp volume $n(V_1 + V_2)$ and the crossing number becomes $n(c-4) + 4$. Cutting along the twice punctured disk bounded by the trivial component and twisting $m$ full twists increases the crossing number by $2m$ while the link complement remains homeomorphic to that of before. The resulting link $J(n,m)$ has cusp crossing density: 

$$d_{cc}(J(n,m)) = \frac{n(V_1 + V_2)}{n(c-4)+4+2m} = \frac{\frac{n}{m}(V_1 + V_2)}{\frac{n}{m}(c-4)+\frac{4}{m}+2}.$$


As $m$ becomes arbitrarily large with respect to $n$, the cusp crossing density approaches 0. If $n$ becomes arbitrarily large relative to $m$, the cusp crossing density approaches $d$. 

Define $$f(t) = \frac{t(V_1 + V_2)}{t(c-4)+2}.$$ Then since f is continuous and can approach both 0 and $d$, for every $y \in (0,d)$, there exists a $t$ such that $f(t) = y$.

 We can choose $\frac{n}{m}$ to be arbitrarily close to $t$ in such a way that $n$ is odd and $m$ is arbitrarily large,  thereby giving cusp crossing density arbitrarily close to $y = f(t)$ for the resulting link. Therefore the cusp crossing densities of two-component links are dense in the interval $[0,d]$.
\end{proof}

We may now find links suitable for the lemma above in order to give specific bounds for intervals of density of cusp crossing density.

\begin{theorem}
Cusp crossing density for two component links is dense in the interval $[0, 1.6923\dots]$.
\end{theorem}

\begin{proof}
Consider the weaving knot $W(10,11)$ with an added trivial component wrapped around its two outermost strands. The result is hyperbolic by \cite{Adams2}.Using the program SnapPy \cite{snap} to calculate $V_1$ and $V_2$ we find that $d= \frac{V_1+V_2}{c-4}= 1.6923$. We may then apply Lemma 4.4 to obtain the result.  
\end{proof}

\begin{theorem}
Cusp crossing density for links is dense in the interval $[0, 2.120\dots]$.
\end{theorem}

\begin{proof} Take the left link L$(7,8)$ from Figure \ref{octahedral1}.  Add a trivial component that forms the boundary of a twice-punctured disk that is punctured by the two outermost concentric components. According to SnapPy, the resulting link has volume $356.69\dots$ and cusp volume of 267.1551, not including the cusp volume of the added component. Excluding the added component, the number of crossings is 126. Then, as in the proof of Lemma \ref{twocomponent}, by taking covers and surgeries on the added component, we can show that cusp densities of links are dense in the interval [0,2.120\dots].
\end{proof}

\end{document}